\documentclass[a4paper,11pt,reqno]{amsart}
\usepackage{amsthm}
\usepackage{amsmath}
\usepackage{enumerate}
\usepackage{amsfonts}
\usepackage{amssymb}
\usepackage{fullpage}
\usepackage{amsmath,amscd}
\usepackage{stmaryrd}
\usepackage{graphicx}
\usepackage{tikz-cd}
\usepackage{array}
\usepackage{amsmath,amssymb,amsfonts,dsfont}
\usepackage[utf8]{inputenc}
\usepackage[english]{babel}
\usepackage[T1]{fontenc}
\usepackage{graphicx}
\usepackage{float}
\usepackage{pdfpages}
\usepackage[a4paper, margin = 3cm, bottom = 3cm]{geometry}
\usepackage{ifpdf}
\usepackage{marginnote}
\usepackage{hyperref}	
\usepackage{tikz-cd}

\usetikzlibrary{calc}
\usetikzlibrary{decorations.pathreplacing,decorations.markings,decorations.pathmorphing}
\usetikzlibrary{positioning,arrows,patterns}
\usetikzlibrary{cd}
\usetikzlibrary{intersections}
\usetikzlibrary{arrows}

\hypersetup{pdfborder=0 0 0,
	    colorlinks=true,
	    citecolor=black,
	    linkcolor=blue,
	    urlcolor=red,
	    pdfauthor={Guillaume Tahar}
	   }

\newcommand{\C}{\mathds{C}}

\def\cal{\mathcal} 
\newtheorem{thm}{Theorem}[section]
\newtheorem{cor}[thm]{Corollary}
\newtheorem{prop}[thm]{Proposition}
\newtheorem{lem}[thm]{Lemma}

\theoremstyle{definition}
\newtheorem{defn}[thm]{Definition}
\theoremstyle{remark}

\theoremstyle{definition}

\theoremstyle{definition}

\theoremstyle{definition}

\numberwithin{equation}{section}
\def\Saff{{\boldsymbol S}}

\def\Sp{{\mathbb S}}

\def\C{{\mathbb C}}
\def\ssm{\smallsetminus}

\def\Caff{{\boldsymbol C}}

\def\Z{{\mathbb Z}}

\def\R{{\mathbb R}}
\def\Uaff{\boldsymbol U}
\def\T{{\rm T}}

\def\Cdil{{\boldsymbol{\cal E}}}
\def\bif{{\cal B}}
\def\reg{{\cal R}}
\def\aut{{\rm Aut}}

\def\and{\quad \text{and}\quad}

\title{Conjugacies of geodesic flows in affine cylinders and tori}

\author{Xavier Buff}
\address[Xavier Buff]{Institut Math\'{e}matique de Toulouse}
\email{xavier.buff@math.univ-toulouse.fr}

\author[G.~Tahar]{Guillaume Tahar}
\address[Guillaume Tahar]{Beijing Institute of Mathematical Sciences and Applications, Huairou District, Beijing, China}
\email{guillaume.tahar@bimsa.cn}

\begin{document}
\maketitle
\setcounter{tocdepth}1

\begin{abstract}
Affine cylinders (genus zero surfaces with two singularities) and affine tori (genus one surfaces without singularities) are among the simplest examples of surfaces endowed with a complex affine structure. Their geodesic flows are particularly tractable. In this article, we provide explicit necessary and sufficient conditions under which the geodesic flows on such surfaces are conjugate, in the topological and in the holomorphic category.
\end{abstract}

\section{Introduction}

The geodesic flow on a surface endowed with a complex affine structure can be viewed as a natural generalization of the directional flows on translation and dilation surfaces. The latter have been extensively studied over the past decades, notably through their deep connections with the theory of interval exchange transformations (see \cite{G17,Zor} for background). Beyond these particularly interesting subclasses, an unexpected link with two-dimensional holomorphic dynamics motivates the study of more general complex affine structures. Indeed, given a homogeneous polynomial vector field on $\C^2$, the projection of real-time trajectories to $\mathbb{CP}^1$ are either points or geodesics with respect to some complex affine structure on $\mathbb{CP}^1$ minus finitely many points (see \cite{AT11,GuillotRebelo}).
\par
In this article, we formulate the problem of conjugacy for geodesic flows on complex affine surfaces and provide a solution, for both topological and holomorphic conjugacies, within a restricted class of affine structures.

%\begin{defn}
%A \textit{translation surface} is a Riemann surface with an atlas $(\zeta_i:U_i\to \C)$ whose change of charts are translations, i.e., such that $\zeta_j =  \zeta_i +\mu_{i,j}$ on $U_i\cap U_j$, with $\mu_{i,j}\in \C$.
%\end{defn}

\begin{defn}
An \textit{affine surface} is a Riemann surface with an atlas $(\zeta_i:U_i\to \C)$ whose change of charts are complex affine maps, i.e., such that $\zeta_j = \lambda_{i,j} \zeta_i +\mu_{i,j}$ on $U_i\cap U_j$, with $(\lambda_{i,j},\mu_{i,j})\in \bigl(\C\ssm \{0\}\bigr)\times \C$.
\end{defn}

If $\Saff$ is an affine surface, we denote by $\Sp$ the underlying Riemann surface and we say that $\Saff$ is an affine structure on $\Sp$. Here are some examples of affine surfaces: 
\begin{itemize}
    \item the affine plane $\Caff$ whose underlying Riemann surface is the complex plane equipped with the canonical chart $z$;
    \item for any $\mu \in \C\ssm \{0\}$, the flat cylinder $\Caff/\mu\Z$ is the quotient of $\Caff$ by the group of translations generated by $z\mapsto z+\mu$; 
    \item for any pair $(\mu,\nu) \in \C^2$ of complex numbers which are $\R$-linearly independent, the flat torus $\Caff/(\mu\Z\oplus \nu \Z)$ is the quotient of $\Caff$ by the group of translations generated by $z\mapsto z+\mu$ and $z\mapsto z+\nu$.
\end{itemize}

\begin{defn}
A map $\phi:\Saff_1\to \Saff_2$ between two affine surfaces $\Saff_1$ and $\Saff_2$ is an \textit{affine map} if for any local affine charts $\zeta_1$ near $s_1\in \Saff_1$ and $\zeta_2$ near $s_2 := \phi(s_1)$ in $\Saff_2$, we locally have that $\zeta_2\circ \phi := \lambda \zeta_1 + \mu$ with $(\lambda,\mu)\in \bigl(\C\ssm \{0\}\bigr)\times \C$.
\end{defn}

\begin{defn}
A \textit{geodesic} in an affine surface $\Saff$ is a curve $\delta:I\to \Saff$ defined on an interval $I\subseteq \R$ which is the restriction of an affine map $\phi:\Uaff\to \Saff$ defined on some neighborhood $\Uaff$ of $I$ in $\Caff$. 
\end{defn}

If $\delta_1:I_1\to \Saff$ and $\delta_2:I_2\to \Saff$ are two geodesics which coincide on $I_1\cap I_2\neq \emptyset$, then the curve $\delta:I_1\cup I_2\to \Saff$ which coincides with $\delta_1$ on $I_1$ and with $\delta_2$ on $I_2$ is a geodesic. 
From now on, we always assume that geodesics are defined on maximal intervals.

Let us denote by $\T\Saff$ the tangent bundle of $\Saff$. 
If $\delta:I\to \Saff$ is a geodesic, we denote as $\dot \delta(t)\in \T_{\delta(t)}\Saff\ssm \{0\}$ the tangent vector to $\delta$ at $t$. We define in such a way a curve $\dot\delta:I\to \T\Saff$ which does not vanish. 
Let us denote by $\T\Saff^*$ the set of nonzero tangent vectors in $\T\Saff$. For any $v\in \T\Saff^*$, there is a unique geodesic $\delta_v:I_v\to \R$ defined on a (maximal) interval containing $0$ such that  $\dot\delta_v(0) = v$. For $t\in I_v$, we set 
\[F_\Saff^t(v) := \dot \delta_v (t)\in \T\Saff^*.\]
This defines a local $\R$-action on $\T\Saff^*$ called the \textit{geodesic flow}. By construction, orbits of the geodesic flow in $\T\Saff^*$ correspond to geodesics in $\Saff$.  Those orbits are leaves of a $1$-real-dimensional oriented foliation of $\T\Saff^*$ generated by the geodesic flow. 

\begin{defn}
The geodesic flow $F_{\Saff_1}^t$ on $\Saff_1$ is {\em topologically/holomorphically conjugate} to the geodesic flow $F_{\Saff_2}^t$ on $\Saff_2$ if there is a homeomorphism/biholomorphism $\Phi:\T\Saff_1^*\to \T\Saff_2^*$ such that for all $v \in \T\Saff_1^*$, we have that $F_{\Saff_1}^t(v)$ is defined if and only if $F_{\Saff_2}^{t}\circ \Phi(v)$ is defined and equal to $\Phi\circ F_{\Saff_1}^{t}(v)$. 
\end{defn}

We would like to understand under which conditions the geodesic flows on two affine surfaces $\Saff_1$ and $\Saff_2$ are topologically/holomorphically conjugate. When $\Saff_1$ and $\Saff_2$ are isomorphic (as affine surfaces), the geodesic flows are holomorphically conjugate. One may think that the situation is rigid and that the converse is true. However, we will see that this is not always the case. 

In this article, we focus on the elementary case of affine cylinders and affine tori. More precisely, let $\Cdil$ be the the \textit{exponential-affine plane}, i.e., the affine structure on $\C$, whose affine charts are the restrictions of the map $\exp(z)$. The map $\exp:\Cdil\to \Caff\ssm\{0\}$ is a universal cover of affine surfaces. The affine automorphisms of $\Caff\ssm\{0\}$ are the maps $z\mapsto \lambda z$ with $\lambda\in \C\ssm \{0\}$. It follows that the group $\aut(\Cdil)$ of affine automorphisms of $\Cdil$ consists of the maps $z\mapsto z+\mu$ with $\mu \in \C$. 

\begin{defn}
An {\em affine cylinder} is an affine surface isomorphic to the quotient surface $\Cdil/\mu \Z$ for some $\mu\in \C\ssm \{0\}$. An {\em affine torus} is an affine surface isomorphic to the quotient $\Cdil/(\mu\Z\oplus \nu\Z)$, where $(\mu,\nu)\in \C^2$ are $\R$-linearly independent.
\end{defn}

Note that the flat cylinder $\Caff/\mu\Z$ and the affine cylinder $\Cdil/\mu\Z$ are affine structures on the same Riemann surface $\C/\mu\Z$, but they are not isomorphic as affine surfaces. Similarly, the flat torus $\Caff/(\mu\Z\oplus \nu \Z)$ and the affine torus $\Cdil/(\mu\Z\oplus \nu \Z)$ are affine structures on the same Riemann surface $\C/(\mu\Z\oplus \nu \Z)$, but they are not isomorphic as affine surfaces.

The main ingredient to characterize conjugacy classes between geodesic flows on affine cylinders or affine tori $\Saff$ is the automorphism $h_\Saff:\Saff\to \Saff$ induced by the automorphism $z\mapsto z+2\pi {\rm i}$ on $\Cdil$.

Our main result is the following. 

\begin{thm}\label{thm:main}
    Let $\Saff_1$ and $\Saff_2$ be affine surfaces obtained as the quotient of $\Cdil$ by a discrete subgroup of $\aut(\Cdil)$. Let $h_{\Saff_1}:\Saff_1\to \Saff_1$ and $h_{\Saff_2}:\Saff_2\to \Saff_2$ be the automorphisms induced by the automorphism $z\mapsto z+2\pi {\rm i}$ on $\Cdil$. The geodesic flow on $\Saff_1$ is topologically/holomorphically conjugate to the geodesic flow on $\Saff_2$ if and only if  $h_{\Saff_1}$ is topologically/holomorphically conjugate to $h_{\Saff_2}$. 
\end{thm}

%Note that if the geodesic flows on such surfaces $\Saff_1$ and $\Saff_2$ are conjugate, then $\Saff_1$ and $\Saff_2$ are homeomorphic. 
%In particular, $\Cdil$ is the unique representative of its conjugacy class. 

%{\color{red} Je ne comprends pas ce que nous avons voulu écrire ici.}

The automorphism $z\mapsto z+2\pi {\rm i}$ on the affine cylinder $\Cdil/\mu\Z$ is holomorphically conjugate to $z\mapsto z+\frac{2\pi {\rm i}}\mu$ on the  cylinder $\C/\Z$. It follows that we have the following classification in the case of affine cylinders. 

\begin{prop}\label{prop:cylinders}
Let $\Saff_1 := \Cdil/\mu_1\Z$ and $\Saff_2:= \Cdil/\mu_2\Z$ be two affine cylinders. The geodesic flows on $\Saff_1$ and $\Saff_2$ are holomorphically conjugate if and only if $\frac{2\pi {\rm i}}{\mu_2}-\frac{2\pi {\rm i}}{\mu_1}\in \Z$ or $\frac{2\pi {\rm i}}{\mu_2}+\frac{2\pi {\rm i}}{\mu_1}\in \Z$. The geodesic flows on $\Saff_1$ and $\Saff_2$ are topologically conjugate if and only if 
\begin{itemize}
    \item either they are holomorphically conjugate;
    \item or both $\Re(\mu_1)\neq 0$ and $\Re(\mu_2)\neq 0$.
\end{itemize}
\end{prop}

\begin{defn}
A geodesic $\delta:I\to \Saff$ is a {\em closed geodesic} if $\delta(t_1) = \delta(t_2)$ and $\dot \delta(t_1) = \lambda \dot \delta(t_2)$ for some $t_1\neq t_2$ in $I$ and some $\lambda \in (0,+\infty)$. The geodesic is {\em periodic} if in addition $\lambda = 1$.
\end{defn}

If $\delta$ is a periodic geodesic, then $\dot \delta$ is a periodic curve in $\T\Saff^*$. It follows that having periodic geodesics is invariant under holomorphic/topological conjugacy of the geodesic flows. It came to us as a surprise that having closed geodesics is not an invariant of holomorphic/topological conjugacy of the geodesic flows. 

Indeed, the affine cylinder $\Cdil/\mu\Z$ contains closed geodesics if and only if $\mu \in \R\ssm \{0\}$. In that case, the curve $\delta:(0,+\infty)\to \Cdil/\mu\Z$ defined by $\delta(t) = \log t$ satisfies $\delta({\rm e}^\mu t) = \delta(t)$ and $\dot \delta ({\rm e}^\mu t) = {\rm e}^\mu \dot \delta(t)$. The previous result shows that the geodesic flows on $\Cdil/\Z$ and $\Cdil/\mu\Z$ with $\mu = \frac{2\pi{\rm i}}{2\pi {\rm i} -1}$ are holomorphically conjugate. However, $\Cdil/\Z$ has closed geodesics and $\Cdil/\mu \Z$ does not have closed geodesics since $\mu\not\in \R\ssm \{0\}$.

In order to describe the classification for affine tori, we need to remind some classical notions. 

\begin{defn}
If $\Sp_1$ and $\Sp_2$ are Riemann surfaces with marked points $(A_1,B_1)\in \Sp_1^2$ and $(A_2,B_2)\in \Sp_2^2$, we say that the marked Riemann surfaces  $(\Sp_1;A_1,B_1)$ and $(\Sp_2;A_2,B_2)$ are {\em biholomorphic} if there is a holomorphic isomorphism $\phi:\Sp_1\to \Sp_2$ such that $\phi(A_1) = A_2$ and $\phi(B_1)=B_2$. 
\end{defn}

\begin{defn}
A diffeomorphism $\phi:\Sp_1\to \Sp_2$ between two complex tori $\Sp_1:=\C/\Gamma_1$ and $\Sp_2:=\C/\Gamma_2$ is {\em real-affine} if it lifts to a diffeomorphism $\Phi:\C\to \C$ of the form $\Phi(z) = L(z) + \mu$ for some $\R$-linear map $L:\C\to \C$ and some complex number $\mu\in \C$.  
\end{defn}

\begin{defn}
Two marked complex tori $(\Sp_1;A_1,B_1)$ and $(\Sp_2;A_2;B_2)$ are {\em real-affine equivalent} if there is a real-affine diffeomorphism $\phi:\Sp_1\to \Sp_2$ such that $\phi(A_1) = A_2$ and $\phi(B_1) = B_2$. 
\end{defn}

\begin{prop}\label{prop:tori}
Let $\Saff_1 := \Cdil/\Gamma_1$ and $\Saff_2:= \Cdil/\Gamma_2$ be two affine tori and consider the associated complex tori $\Sp_1:=\C/\Gamma_1$ and $\Sp_2:=\C/\Gamma_2$. For $j\in \{1,2\}$, let $A_j$ be the class of $0$ in $\Sp_j$ and let $B_j$ be the class of $2\pi {\rm i}$ in $\Sp_j$.
The geodesic flows on $\Saff_1$ and $\Saff_2$ are holomorphically conjugate if and only if the marked tori $(\Sp_1;A_1,B_1)$ and $(\Sp_2;A_2,B_2)$ are biholomorphic. 
The geodesic flows on $\Saff_1$ and $\Saff_2$ are topologically conjugate if and only if the marked tori $(\Sp_1;A_1,B_1)$ and $(\Sp_2;A_2,B_2)$ are real-affine equivalent.
\end{prop}

\paragraph{Organization of the paper:}
\begin{itemize}
    \item In Section~\ref{sec:Geodesicflow}, we describe the geodesic flow in the exponential-affine plane $\Cdil$ and its quotients.
    \item In Section~\ref{sec:Invariant}, we show that for any quotient $\Cdil/\Gamma$, the topological/holomorphic conjugacy class of the automorphism induced by $z\mapsto z+2\pi {\rm i}$ on $\Cdil$ is an invariant of the topological/holomorphic conjugacy class of the geodesic flow. 
    \item In Section~\ref{sec:Explicit}, we construct topological/holomorphic conjugacies of the geodesics flows having the same topological/holomorphic invariant and prove Theorem~\ref{thm:main}.
    \item In Section~\ref{sec:cylinders}, we deduce Proposition~\ref{prop:cylinders} from Theorem~\ref{thm:main}.
    \item In Section~\ref{sec:tori}, we deduce Proposition~\ref{prop:tori} from Theorem~\ref{thm:main}.
\end{itemize}

\paragraph{\bf Acknowledgements.}
This research was conducted during a semester on Holomorphic Dynamics and Geometry of Surfaces funded by the ANR LabEx CIMI (grant
ANR-11-LABX-0040) within the French State program “Investissements d’Avenir”. 
The research of both authors is supported by the French National Research Agency under the project TIGerS (ANR-24-CE40-3604).
Research by G.T. is also supported by the Beijing Natural Science Foundation (Grant IS23005).

\section{Description of the geodesic flow}\label{sec:Geodesicflow}

Remember that $\Caff$ is the canonical affine structure on $\C$. The tangent space $\T\Caff$ naturally identifies to $\C^2$, the canonical projection $\T\Caff\to \Caff$ identifying with $(z,u)\mapsto z$. The geodesic flow $F^t_\Caff$ is globally defined on $\T\Caff^{\ast}$ and is given by 
\[F^t_\Caff (z,u) = (z+tu,u).\]
However, on an arbitrary affine surface $\Saff$, the geodesic flow is usually not globally defined: $\T\Saff^*$ decomposes into 
\begin{itemize}
    \item the regular locus 
    \[\reg(\Saff):=\bigl\{v\in \T\Saff^*~:~F^t_\Saff(v)\text{ is defined for all }t\in \R\bigr\}\] and 
    \item the bifurcation locus 
    \[\bif(\Saff):=\T\Saff^*\ssm \reg(\Saff)\] which is the complement of the regular locus.
\end{itemize}
The restriction of the geodesic flow to the regular locus $\reg(\Saff)$ is complete by definition. 

Furthermore, given $\tau\in (0,+\infty)$ (respectively $\tau\in (-\infty,0)$), we may also consider the set $\bif^\tau(\Saff)\subset \bif(\Saff)$ of tangent vectors such that $F^t_\Saff(v)$ is defined for $t\in [0,\tau)$ (respectively for $t\in (\tau,0]$) but not for $t=\tau$.

The regular locus and the bifurcation locus are conjugacy invariants: if $\Phi:\T\Saff_1^*\to \T\Saff_2^*$ conjugates the geodesic flow on $\Saff_1$ to the geodesic flow on $\Saff_2$, then 
\[\Phi\bigl(\reg(\Saff_1)\bigr) = \reg(\Saff_2), \quad \Phi\bigl(\bif(\Saff_1)\bigr) = \bif(\Saff_2).\]
In addition, 
\[\forall \tau\in \R\ssm \{0\},\quad \Phi\bigl(\bif^\tau(\Saff_1)\bigr) = \bif^\tau(\Saff_2).\]

\subsection{The geodesic flow on $\Cdil$}

Remember that $\Cdil$ is the affine structure on $\C$, whose affine charts are the restrictions of the map $\exp(z)$. Note that $z:\Cdil\to \C$ is a holomorphic isomorphism, but not an affine isomorphism. Nevertheless, it induces a natural holomorphic isomorphism between the tangent space $\T\Cdil$ and $\C^2$, the canonical projection $\T\Cdil \to \Cdil$ identifying to $(z,u)\mapsto z$. 

By definition of $\Cdil$, the map $\exp:\Cdil \to \Caff\ssm \{0\}$ is a universal cover of affine surfaces. The corresponding bundle map $\T\Cdil\to \T(\Caff\ssm \{0\})$ identifies to $(z,u)\mapsto \bigl(\exp(z),\exp(z) u\bigr)$. 
We have that 
\[\bigl(\exp(z) + t \exp(z) u, \exp(z)u) = \left(\exp \bigl(z + \log(1+tu)\bigr),\exp \bigl(z + \log(1+tu)\bigr)\frac{u}{1+tu}\right).\]
It follows that the geodesic flow on $\Cdil$ identifies with 
\[F^t_\Cdil:(z,u)\mapsto \left(z+\log(1+tu),\frac{u}{1+tu}\right).\]
As a consequence, we have the identifications
\[\forall \tau \in \R\ssm \{0\},\quad \bif^\tau(\Cdil)\simeq \bigl\{(z,u)\in \T\Cdil~:~1+\tau u = 0\bigr\} = \left\{(z,u)\in \T\Cdil~:~u = -\frac1\tau\right\},\]
\[\bif(\Cdil)\simeq \bigl\{(z,u)\in \T\Cdil~:~u\in \R\ssm \{0\}\bigr\}\and \reg(\Cdil) \simeq \bigl\{(z,u)\in \T\Cdil~:~u\in \C\ssm \R\bigr\}.\]
In particular, the regular locus $\reg(\Cdil)$ is open and has two connected components 
\[\reg^+(\Cdil)\simeq \bigl\{(z,u)\in \T\Cdil~:~\Im(u)>0\bigr\}\and \reg^-(\Cdil)\simeq \bigl\{(z,u)\in \T\Cdil~:~\Im(u)<0\bigr\}.\]
Note that each connected component $\reg^\pm(\Cdil)$ is invariant by the geodesic flow. 
In addition, the bifurcation locus $\bif(\Cdil)$ is a real $3$-dimensional manifold foliated by the affine surfaces 
\[\Cdil^\tau:= \bif^\tau (\Cdil), \quad \tau \in \R\ssm \{0\},\] and the canonical projection $\T\Cdil\to \Cdil$ induces an affine isomorphism  $\Cdil^\tau\to \Cdil$ for each $\tau \in \R\ssm \{0\}$. 

Let us now assume that $\tau_1\in (-\infty,0)$ and $\tau_2\in (0,+\infty)$ and set $t:=\tau_2-\tau_1$. If $v:=(z,u)\in \Cdil^{\tau_2}$, then $F_\Cdil^t(v)$ is not defined. However, $F_\Cdil^t(v')$ is well defined for any $v'\in \reg(\Cdil)$ and $F_\Cdil^t(v')$ has a limit $F^{t,+}_\Cdil(v)\in \Cdil^{\tau_1}$ as $v'\to v$ with $v'\in \reg^+(\Cdil)$. This limit is given by 
\[F^{t,+}_\Cdil\left(z,-\frac1{\tau_2}\right) = \left(z + \log \left(-\frac{\tau_1}{\tau_2}\right) +\pi {\rm i}, -\frac1{\tau_1}\right).\]
Similarly, $F_\Cdil^t(v')$ has a limit $F^{t,-}_\Cdil(v)\in \Cdil^{\tau_1}$ as $v'\to v$ with $v'\in \reg^-(\Cdil)$. This limit is given by 
\[F^{t,-}_\Cdil\left(z,-\frac1{\tau_2}\right) = \left(z + \log \left(-\frac{\tau_1}{\tau_2}\right) -\pi {\rm i}, -\frac1{\tau_1}\right).\]

\subsection{The geodesic flow on quotients of $\Cdil$}

The group $\aut(\Cdil)$ of affine automorphisms of $\Cdil$ is abelian and consists of the affine maps $z\mapsto z+\mu$ with $\mu\in \C$. A discrete subgroup of $\aut(\Cdil)$ is therefore either trivial, or of rank $1$ or $2$. Let $\Gamma$ be such a subgroup and consider the affine surface $\Saff:=\Cdil/\Gamma$.  

If $\gamma\in \Gamma$ and $(z,u)\in \T\Cdil$, then ${\rm D}\gamma(z,u) = (\gamma(z),u)$. In particular, the regular locus $\reg(\Cdil)$ and the bifurcation locus $\bif(\Cdil)$ are invariant by the action of $\Gamma$, as well as $\reg^\pm(\Cdil)$ and $\Cdil^\tau$ for $\tau\in \R\ssm \{0\}$. We have that 
\[\reg(\Saff) = \reg(\Cdil)/\Gamma,\quad \bif(\Saff) = \bif(\Cdil)/\Gamma\]
and
\[\forall \tau \in \R\ssm \{0\},\quad \Saff^\tau:=\bif^\tau(\Saff) = \Cdil^\tau/\Gamma.\]
Note that the regular locus $\reg(\Saff)$ has two connected components $\reg^\pm(\Saff):= \reg^\pm(\Cdil)/\Gamma$. And each connected component is invariant by the geodesic flow $F^t_\Saff$. 
In addition, the bifurcation locus $\bif(\Saff)$ is a real $3$-dimensional manifold foliated by the affine surfaces $\Saff^\tau$ and the canonical projection $\pi_\Saff:\T\Saff\to \Saff$ induces an affine isomorphism  $\Saff^\tau \to \Saff$ for each $\tau\in \R\ssm \{0\}$.

Assume $\tau_1\in (-\infty,0)$ and $\tau_2\in (0,+\infty)$. Set $t:=\tau_2-\tau_1$. Then, the affine isomorphism
$F_\Cdil^{t,\pm}:\Cdil^{\tau_2}\to \Cdil^{\tau_1}$ induces an affine isomorphism
$F_\Saff^{t,\pm}:\Saff^{\tau_2}\to \Saff^{\tau_1}$. We have that 
\[\forall v\in \Saff^{\tau_1},\quad F_\Saff^{t,\pm}(v) = \lim_{v'\to v\atop v'\in \reg^\pm(\Saff)} F^t_\Saff(v').\]

\section{Conjugacy invariants}\label{sec:Invariant}

We still assume that $\Gamma$ is a discrete subgroup of $\aut(\Cdil)$ and we consider the affine surface $\Saff:=\Cdil/\Gamma$. Let $h_\Saff^\pm:\Saff\to \Saff$ be the affine automorphism of $\Saff$ induced by the affine automorphism $z\mapsto z\pm 2\pi {\rm i}$ of $\Cdil$. We have that $h_{\Saff}^+ = h_\Saff$ and $h_\Saff^- = h_\Saff^{-1}$. 
Note that the derivative ${\rm D}h_\Saff^\pm:\T\Saff\to \T\Saff$ is an automorphism which preserves $\reg(\Saff)$ and $\bif(\Saff)$. It even preserves $\reg^+(\Saff)$, $\reg^-(\Saff)$ and $\bif^\tau(\Saff)$ for all $\tau\in \R\ssm \{0\}$. 

\begin{prop}
Let $\Gamma_1$ and $\Gamma_2$ be discrete subgroups of $\aut(\Cdil)$. Set $\Saff_1 := \Cdil/\Gamma_1$ and $\Saff_2 := \Cdil/\Gamma_2$. If $\Phi:\T\Saff_1^*\to \T\Saff_2^*$ is a conjugacy between the geodesic flows, then 
\begin{itemize}
\item $\Phi\bigl(\reg(\Saff_1)\bigr) = \reg(\Saff_2)$ and $\Phi\bigl(\bif(\Saff_1)\bigr) = \bif(\Saff_2)$;
\item if $\Phi\bigl(\reg^+(\Saff_1)\bigr) = \reg^+(\Saff_2)$, then $\Phi\circ {\rm D}h_{\Saff_1}^+ = {\rm D}h_{\Saff_2}^+\circ \Phi$ on $\bif(\Saff_1)$; 
\item if $\Phi\bigl(\reg^+(\Saff_1)\bigr) = \reg^-(\Saff_2)$, then $\Phi\circ {\rm D}h_{\Saff_1}^+ = {\rm D}h_{\Saff_2}^-\circ \Phi$ on $\bif(\Saff_1)$.
\end{itemize}
\end{prop}

\begin{proof}
Assume $v\in \T\Saff_1^*$. 
Since $\Phi$ conjugates the geodesic flow on $\Saff_1$ to the geodesic flow on $\Saff_2$, $F^t_{\Saff_1}(v)$ is defined for all $t\in \R$ if and only $F^t_{\Saff_2}\circ \Phi(v)$ is defined for all $t\in \R$. In other words, $v\in \reg(\Saff_1)$ if and only if $\Phi(v)\in \reg(\Saff_2)$, and so $v\in \bif(\Saff_1)$ if and only if $\Phi(v)\in \bif(\Saff_2)$

Let us now assume that $\Phi\bigl(\reg^+(\Saff_1)\bigr) = \reg^+(\Saff_2)$. Assume $\tau_1\in (-\infty,0)$ and $\tau_2\in (0,+\infty)$. Set $t := \tau_2-\tau_1$ and fix $v\in \Saff_1^{\tau_2}$. If $v'\in \reg^\pm(\Saff_1)$, then $\Phi(v')\in \reg^\pm(\Saff_2)$. In addition, as $v'$ tends to $v\in \Saff_1^{\tau_2}$, we have that $\Phi(v')$ tends to $\Phi(v)\in \Saff_2^{\tau_2}$. As a consequence, 
\[\Phi\circ F_{\Saff_1}^{t,\pm}(v) = \lim_{v'\to v\atop v'\in \reg^\pm(\Saff_1)} \Phi\circ F^t_{\Saff_1}(v') = \lim_{v'\to v\atop v'\in \reg^\pm(\Saff_1)} F^t_{\Saff_2}\circ \Phi(v') = F_{\Saff_2}^{t,\pm}\circ \Phi(v).\]
As a consequence, we have the following commutative diagrams:
\[
\begin{tikzcd}
\Saff_1^{\tau_2} \arrow[r, "\Phi"] \arrow[d, "F_{\Saff_1}^{t,+}"'] \arrow[dd,"{\rm D}h_{\Saff_1}^+"',bend right=80]
& \Saff_2^{\tau_2} \arrow[d, "F_{\Saff_2}^{t,+}"] \arrow[dd,"{\rm D}h_{\Saff_2}^+",bend left=80] \\
\Saff_1^{\tau_1} \arrow[r, "\Phi"] 
& \Saff_2^{\tau_1} \\
\Saff_1^{\tau_2} \arrow[r, "\Phi"] \arrow[u, "F_{\Saff_1}^{t,-}"]
& \Saff_2^{\tau_2} \arrow[u, "F_{\Saff_2}^{t,-}"']
\end{tikzcd}
\and
\begin{tikzcd}
\Saff_1^{\tau_1} \arrow[r, "\Phi"] \arrow[dd,"{\rm D}h_{\Saff_1}^+"',bend right=80]
& \Saff_2^{\tau_1} \arrow[dd,"{\rm D}h_{\Saff_2}^+",bend left=80] \\
\Saff_1^{\tau_2} \arrow[r, "\Phi"] \arrow[u, "F_{\Saff_1}^{t,-}"] \arrow[d, "F_{\Saff_1}^{t,+}"']
& \Saff_2^{\tau_2} \arrow[u, "F_{\Saff_2}^{t,-}"'] \arrow[d, "F_{\Saff_2}^{t,+}"] \\
\Saff_1^{\tau_1} \arrow[r, "\Phi"]  
& \Saff_2^{\tau_1}.
\end{tikzcd}
\]
This proves that $\Phi\circ {\rm D}h_{\Saff_1}^+ = {\rm D}h_{\Saff_2}^+\circ \Phi$ on $\Saff_1^{\tau_2}$ and on $\Saff_1^{\tau_1}$. 
Since this holds for any $\tau_1\in (-\infty,0)$ and any $\tau_2\in (0,+\infty)$, this shows that $\Phi\circ {\rm D}h_{\Saff_1}^+ = {\rm D}h_{\Saff_2}^+\circ \Phi$ on $\bif(\Saff_1)$ as required. 

Let us finally assume that $\Phi\bigl(\reg^+(\Saff_1)\bigr) = \reg^-(\Saff_2)$. Again, assume $\tau_1\in (-\infty,0)$ and $\tau_2\in (0,+\infty)$. Set $t := \tau_2-\tau_1$ and fix $v\in \Saff_1^{\tau_2}$. Arguing as above, we deduce that we have the following commutative diagrams:
\[
\begin{tikzcd}
\Saff_1^{\tau_2} \arrow[r, "\Phi"] \arrow[d, "F_{\Saff_1}^{t,+}"'] \arrow[dd,"{\rm D}h_{\Saff_1}^+"',bend right=80]
& \Saff_2^{\tau_2} \arrow[d, "F_{\Saff_2}^{t,-}"] \arrow[dd,"{\rm D}h_{\Saff_2}^-",bend left=80] \\
\Saff_1^{\tau_1} \arrow[r, "\Phi"] 
& \Saff_2^{\tau_1} \\
\Saff_1^{\tau_2} \arrow[r, "\Phi"] \arrow[u, "F_{\Saff_1}^{t,-}"]
& \Saff_2^{\tau_2} \arrow[u, "F_{\Saff_2}^{t,+}"']
\end{tikzcd}
\and
\begin{tikzcd}
\Saff_1^{\tau_1} \arrow[r, "\Phi"] \arrow[dd,"{\rm D}h_{\Saff_1}^+"',bend right=80]
& \Saff_2^{\tau_1} \arrow[dd,"{\rm D}h_{\Saff_2}^-",bend left=80] \\
\Saff_1^{\tau_2} \arrow[r, "\Phi"] \arrow[u, "F_{\Saff_1}^{t,-}"] \arrow[d, "F_{\Saff_1}^{t,+}"']
& \Saff_2^{\tau_2} \arrow[u, "F_{\Saff_2}^{t,+}"'] \arrow[d, "F_{\Saff_2}^{t,-}"] \\
\Saff_1^{\tau_1} \arrow[r, "\Phi"]  
& \Saff_2^{\tau_1},
\end{tikzcd}
\]
which shows that $\Phi\circ {\rm D}h_{\Saff_1}^+ = {\rm D}h_{\Saff_2}^-\circ \Phi$ on $\bif(\Saff_1)$ as required. 
\end{proof}

\begin{cor}
Let $\Saff_1 = \Cdil/\Gamma_1$ and $\Saff_2 = \Cdil/\Gamma_2$ be affine quotients of $\Cdil$. If the geodesic flow on $\Saff_1$ is topologically/holomorphically conjugate to the geodesic flow on $\Saff_2$, then the automorphisms $h_{\Saff_1}:\Saff_1\to \Saff_1$ is topologically/holomorphically conjugate to the automorphism $h_{\Saff_2}:\Saff_2\to \Saff_2$. 
\end{cor}

\begin{proof}
Let $\Saff_1 = \Cdil/\Gamma_1$ and $\Saff_2 = \Cdil/\Gamma_2$ be affine quotients of $\Cdil$ and assume $\Phi:\T\Saff_1^*\to \T\Saff_2^*$ is a topological/holomorphic conjugacy between the geodesic flows. 

Let $\iota:\Saff_2\to \Saff_2$ be the involution induced by $z\mapsto -z$ on $\Cdil$. Note that $\iota$ is not an affine automorphism, but it is a holomorphic automorphism. In addition, $\iota$ conjugates $h_{\Saff_2}^-:\Saff_2\to \Saff_2$ to $h_{\Saff_2}^+:\Saff_2\to \Saff_2$. It is therefore enough to prove the existence of a topological/holomorphic conjugacy $\phi:\Saff_1\to \Saff_2$ between $h_{\Saff_1}^+:\Saff_1\to \Saff_1$ and either $h_{\Saff_2}^+:\Saff_2\to \Saff_2$ or $h_{\Saff_2}^-:\Saff_2\to \Saff_2$. 

Fix $\tau\in \R\ssm \{0\}$. On the one hand, $\Phi$ restricts to a homeomorphism $\Saff_1^\tau\to \Saff_2^\tau$ which conjugates ${\rm D}h_{\Saff_1}^+$ to either ${\rm D}h_{\Saff_2}^+$ or to ${\rm D}h_{\Saff_2}^-$. On the other hand, the canonical projections $\T\Saff_1\to \Saff_1$ and $\T\Saff_2\to \Saff_2$ restrict to affine automorphisms $\Saff_1^\tau\to \Saff_1$ and $\Saff_2^\tau\to \Saff_2$. Let $\phi:\Saff_1\to \Saff_2$ be the homeomorphism induced by $\Phi:\Saff_1^\tau\to \Saff_2^\tau$. Then, $\phi$ conjugates $h_{\Saff_1}^+$ to either $h_{\Saff_2}^+$ or to $h_{\Saff_2}^-$. If $\Phi$ is holomorphic, then $\phi$ is holomorphic too. 
\end{proof}

\begin{lem}\label{lem:Inversion}
Given a pair of affine cylinders or tori $\Saff_1$ and $\Saff_2$, if there exists homeomorphisms (resp. biholomorphisms) from $\Saff_1$ to $\Saff_2$ that conjugate $h_{\Saff_1}$ to $h_{\Saff_2}$, then there exists homeomorphisms (resp. biholomorphisms) that conjugate $h_{\Saff_1}$ to $h_{\Saff_2}^{-1}$.
\end{lem}

\begin{proof}
We just have to observe that $\sigma:\Cdil\longrightarrow\Cdil$ where $\sigma(u)=-u$ induces a biholomorphism $\tilde{\sigma}$ on any quotient surface $\Cdil/\Gamma$ that conjugates $h_{\Saff}$ with $h_{\Saff}^{-1}$.
\end{proof}

\section{Existence of conjugacies}\label{sec:Explicit}

Consider two affine surface $\Saff_1 = \Cdil/\Gamma_1$ and $\Saff_2 = \Cdil/\Gamma_2$. Let $\pi_1:\Cdil \to \Saff_1:=\Cdil/\Gamma_1$ and $\pi_2:\Cdil \to \Saff_2:=\Cdil/\Gamma_2$ be the canonical projections. Those are universal coverings.

Note that the group of automorphisms $\aut(\Cdil)$ is abelian. As a consequence, for each $\mu\in \C$, the automorphism $H_\mu:z\mapsto z+\mu$ of $\Cdil$ induces automorphisms $h_{1,\mu}\in \aut(\Saff_1)$ and $h_{2,\mu}\in \aut(\Saff_2)$ such that 
\[\pi_1\circ H_\mu = h_{1,\mu}\circ \pi_1\and \pi_2\circ H_\mu = h_{2,\mu}\circ \pi_2.\]
Note that 
\[\forall (\mu,\nu)\in \C^2,\quad h_{1,\mu+\nu} = h_{1,\mu}\circ h_{1,\nu}\and h_{2,\mu+\nu} = h_{2,\mu}\circ h_{2,\nu}.\]

In this section, we assume that $h_1:=h_{1,2\pi {\rm i}}:\Saff_1\to \Saff_1$ is topologically/holomorphically conjugate to $h_2:=h_{2,2\pi {\rm i}}:\Saff_2\to \Saff_2$, and we deduce that the geodesic flow on $\Saff_1$ is topologically/holomorphically conjugate to the geodesic flow on $\Saff_2$. 
So, let us assume that $\phi:\Saff_1\to \Saff_2$ is a homeomorphism (possibly a holomorphic isomorphism) which conjugates $h_1:\Saff_1\to \Saff_1$ to $h_2:\Saff_2\to \Saff_2$. 

Given $\mu\in \C$, let $\phi_\mu:\Saff_1\to \Saff_2$ be the homeomorphism defined by 
\[\phi_\mu := h_{2,\mu}^{-1}\circ \phi\circ h_{1,\mu}:\Saff_1\to \Saff_2.\]
Note that $\phi_1 = \phi$ and for all $\mu\in \C$,
\[\forall \mu\in \C, \quad \phi_{\mu+2\pi{\rm i}} = h_{2,\mu}^{-1}\circ h_2^{-1}\circ \phi\circ h_1\circ h_{1,\mu} = h_{2,\mu}^{-1}\circ \phi\circ h_{1,\mu} = \phi_\mu.\]
As a consequence, we can define a homeomorphism $\Psi:\T\Saff_1^*\to \T\Saff_2^*$ such that 
    \[\forall (z,u)\in \T\Saff_1^*,\quad \Psi(z,u) = \bigl(\phi_{\log u}(z),u\bigr) = \bigl(\phi(z+\log u)-\log u,u\bigr).\]
The definition does not depend on the choice of branch of $\log u$. In addition, if $\phi$ is holomorphic, then $\Psi:\T\Saff_1^*\to \T\Saff_2^*$ is a holomorphic isomorphism. 

We will now show that $\Psi$ conjugates the geodesic flow on $\Saff_1$ to the geodesic flow on $\Saff_2$, which completes the proof of Theorem \ref{thm:main}. This is immediate since
\begin{eqnarray*}
    \Psi\circ {\cal F}^t_{\Saff_1}(z,u) &=& \Psi\left(z + \log(1+tu),\frac{u}{1+tu}\right)\\
    &=& \left(\phi\left(z + \log(1+tu)+\log\frac{u}{1+tu}\right) - \log \frac{u}{1+tu},\frac{u}{1+tu}\right) \\
    & =&  \left(\phi\left(z + \log u\right) - \log u + \log (1+tu),\frac{u}{1+tu}\right)\\
    &=& {\cal F}^t_{\Saff_2}\left(\phi\left(z + \log u\right) - \log u ,u\right) \\
    &=& {\cal F}^t_{\Saff_2}\circ \Psi(z,u).
\end{eqnarray*}

\section{The case of affine cylinders}\label{sec:cylinders}

In this section, we prove Proposition~\ref{prop:cylinders}. 
Consider the affine cylinders $\Saff_1:=\Cdil/\mu_1\Z$ and $\Saff_2:=\Cdil/\mu_2\Z$ with $\mu_1$ and $\mu_2$ in $\C\ssm \{0\}$. 

Every homeomorphism  $\phi:\Saff_1\to \Saff_2$ lifts to a homeomorphism $\Phi:\Cdil\to \Cdil$ which satisfies $\Phi(z+\mu_1) = \Phi(z) \pm \mu_2$. 
In addition, such a $\phi$ conjugates $h_{\Saff_1}$ to $h_{\Saff_2}$ if and only if $\Phi(z+2\pi{\rm i}) = \Phi(z) + 2\pi {\rm i} + k\mu_2$ for some integer $k\in \Z$. 

\subsection{Holomorphic conjugacy}

If $\phi$ is a holomorphic isomorphism, then 
\[\Phi(z) = \pm \frac{\mu_2}{\mu_1} z + \tau \quad \text{with}\quad \tau\in \C. \]
And $\phi$ conjugates $h_{\Saff_1}$ to $h_{\Saff_2}$ if and only if 
\[\forall z\in \Cdil, \quad \pm\frac{\mu_2}{\mu_1} (z +2\pi {\rm i}) +\tau = \pm\frac{\mu_2}{\mu_1} z + \tau + 2\pi {\rm i} + k\mu_2\quad \text{with}\quad k\in \Z,\]
i.e., if and only if 
\[\frac{2\pi {\rm i}}{\mu_2} \pm \frac{2\pi {\rm i}}{\mu_1}\in \Z.\]
As a consequence, if there exists a holomorphic isomorphism $\phi:\Saff_1\to \Saff_2$ which conjugates $h_{\Saff_1}$ to $h_{\Saff_2}$, then $\frac{2\pi {\rm i}}{\mu_2} \pm \frac{2\pi {\rm i}}{\mu_1}\in \Z$. 
Conversely, if $\frac{2\pi {\rm i}}{\mu_2} - \frac{2\pi {\rm i}}{\mu_1}\in \Z$, then the holomorphic isomorphism $\phi:\Saff_1\to \Saff_2$ defined by $\phi(z) = \frac{\mu_2}{\mu_1} z$ conjugates $h_{\Saff_1}$ to $h_{\Saff_2}$ and if $\frac{2\pi {\rm i}}{\mu_2} + \frac{2\pi {\rm i}}{\mu_1}\in \Z$, then the holomorphic isomorphism $\phi:\Saff_1\to \Saff_2$ defined by $\phi(z) = -\frac{\mu_2}{\mu_1} z$ conjugates $h_{\Saff_1}$ to $h_{\Saff_2}$. 

According to Theorem \ref{thm:main}, the geodesic flow on $\Saff_1$ is holomorphically conjugate to the geodesic flow on $\Saff_2$ if and only if $\frac{2\pi {\rm i}}{\mu_2} \pm \frac{2\pi {\rm i}}{\mu_1}\in \Z$. 

\subsection{Topological conjugacy}

Let us first assume that $\phi:\Saff_1\to \Saff_2$ is a homeomorphism which conjugates $h_{\Saff_1}$ to $h_{\Saff_2}$. Let $\Phi:\Cdil\to \Cdil$ be a lift which satisfies 
\[\Phi(x+\mu_1) = \Phi(x)\pm \mu_2\quad \text{and}\quad \Phi(x+2\pi{\rm i}) = \Phi(x) + 2\pi {\rm i} + k\mu_2\quad \text{with}\quad k\in \Z.\]
If $\Re(\mu_1) = 0$, then $\alpha := \frac{2\pi {\rm i}}{\mu_1}\in \R$ and we can find sequences $(p_n)_{n\geq 0}$ and $(q_n)_{n\geq 0}$ such that $q_n \alpha - p_n\to 0$ as $n\to +\infty$, i.e., $q_n2\pi {\rm i} - p_n \mu_1\to 0$ as $n\to +\infty$. In that case 
\[\Phi(0)+q_n(2\pi {\rm i} + k \mu_2)\mp p_n \mu_2 = \Phi(q_n2\pi {\rm i} - p_n \mu_1) \underset{n\to +\infty}\longrightarrow \Phi(0).\]
We deduce that 
\[\frac{2\pi {\rm i}}{\mu_1} = \alpha = \lim_{n\to +\infty} \frac{p_n}{q_n} = \pm \frac{2\pi {\rm i}}{\mu_2} + k.\]
Thus, $\Re(\mu_2) = 0$ and 
\[\frac{2\pi {\rm i}}{\mu_2} \mp \frac{2\pi {\rm i}}{\mu_1} = - k\in \Z.\]
Similarly, if $\Re(\mu_2) = 0$, then $\Re(\mu_1) = 0$ and $\frac{2\pi {\rm i}}{\mu_2} \mp \frac{2\pi {\rm i}}{\mu_1} \in \Z$. Thus, according to Theorem \ref{thm:main} and the previous discussion regarding holomorphic conjugacies, if the geodesic flows on $\Saff_1$ and $\Saff_2$ are topologically conjugate, then either they are holomorphically conjugate, or both $\Re(\mu_1)\neq 0$ and $\Re(\mu_2)\neq 0$.

Let us now assume that both $\Re(\mu_1)\neq 0$ and $\Re(\mu_2)\neq 0$. We will now exhibit a homeomorphism $\phi:\Saff_1\to \Saff_2$ which conjugates $h_{\Saff_1}$ to $h_{\Saff_2}$. %Replacing $\mu_1$ and $\mu_2$ by $-\mu_1$ and $-\mu_2$ if necessary (which does not change $\Saff_1$ and $\Saff_2$, we may assume that $\Re(\mu_1)>0$ and $\Re(\mu_2)>0$. 
Note that $(\mu_1,2\pi {\rm i})$ and $(\mu_2,2\pi {\rm i})$ are bases of $\Cdil\simeq \C$ considered as a $\R$-vector space. We may therefore define a $\R$-linear homeomorphism $\Phi:\Cdil \to \Cdil$ by 
\[\forall (x,y)\in \R^2,\quad \Phi(x\mu_1 + y2\pi {\rm i}) = x\mu_2 + y2\pi {\rm i}.\]
Note that for all $z\in \Cdil$, we have that $\Phi(z+\mu_1) = \Phi(z) + \mu_2$, so that $\Phi$ induces a homeomorphism $\phi:\Saff_1\to \Saff_2$. In addition, $\Phi(z+2\pi{\rm i}) = \Phi(z) + 2\pi {\rm i}$, so that $\phi$ conjugates $h_{\Saff_1}$ to $h_{\Saff_2}$. According to Theorem \ref{thm:main}, this proves that when both $\Re(\mu_1)\neq 0$ and $\Re(\mu_2)\neq 0$, the geodesic flows on $\Saff_1$ and $\Saff_2$ are topologically conjugate. 

\section{The case of affine tori}\label{sec:tori}

In this section, we deduce Proposition~\ref{prop:tori} from Theorem~\ref{thm:main}. Let $\Gamma_1$ and $\Gamma_2$ be two lattices in $\C$. Consider the affine tori $\Saff_1:=\Cdil/\Gamma_1$ and $\Saff_2:=\Cdil/\Gamma_2$  and the complex tori $\Sp_1:=\Caff/\Gamma_1$ and $\Sp_2:=\Caff/\Gamma_2$. For $j\in \{1,2\}$, let $A_j$ be the class of $0$ in $\Sp_j$ and let $B_j$ be the class of $2\pi {\rm i}$ in $\Sp_j$. 

For $j\in \{1,2\}$, the affine automorphism $h_{\Saff_j}:\Saff_j\to \Saff_j$ may be considered as a translation on the complex torus $\Sp_j$; in addition, the group of translations of $\Sp_j$ is abelian and acts transitively on $\Sp_j$. As a consequence, $h_{\Saff_1}:\Saff_1\to \Saff_1$ is topologically/holomorphically conjugate to $h_{\Saff_1}:\Saff_1\to \Saff_1$ if and only if there exists a homeomorphism/biholomorphism $\phi:\Sp_1\to \Sp_2$ such that $\phi\circ h_{\Saff_1} = h_{\Saff_2}\circ \phi$ and $\phi(A_1) = A_2$. 

In addition, a homeomorphism/biholomorphism $\phi:\Sp_1\to \Sp_2$ sending $A_1$ to $A_2$ lifts to a homeomorphism/biholomorphism $\Phi:\C\to \C$ which fixes $0$, sends any pair of generators $(\mu_1,\nu_1)$ of $\Gamma_1$ to a pair of generators $(\mu_2,\nu_2)$ of $\Gamma_2$, and satisfies $\Phi(z+\mu_1) = \Phi(z) + \mu_2$ and $\Phi(z+\nu_1) = \Phi(z) + \nu_2$ for all $z\in \C$. In this situation, $\phi$ conjugates $h_{\Saff_1}$ to $h_{\Saff_2}$ if and only if there exists $\gamma \in \Gamma_2$ such that $\Phi(z+2\pi{\rm i}) = \Phi(z) + 2\pi {\rm i} + \gamma$ for all $z\in \C$.

\subsection{Holomorphic conjugacy}

Assume $\phi:\Sp_1\to \Sp_2$ is a holomorphic isomorphism sending $A_1$ to $A_2$ and let $\Phi:\C\to \C$ be the lift fixing $0$. Then, $\Phi(z) = \alpha z$ for some $\alpha\in \C\ssm \{0\}$.

If $\phi$ conjugates $h_{\Saff_1}$ to $h_{\Saff_2}$, then there exists $\gamma \in \Gamma_2$ such that $\Phi(z+2\pi{\rm i}) = \Phi(z) + 2\pi {\rm i} + \gamma$ for all $z\in \C$. In particular, $\Phi(2\pi{\rm i}) = 2\pi {\rm i} + \gamma$, and so, $\phi(B_1) = B_2$. 

Conversely, if $\phi$ sends $B_1$ to $B_2$, then $\alpha2\pi {\rm i} = \Phi(2\pi {\rm i}) = 2\pi{\rm i} + \gamma$ for some $\gamma\in \Gamma_2$. In that case, we have that
\[\forall z\in \C, \quad \Phi(z+2\pi{\rm i}) = \alpha(z+2\pi {\rm i}) = \alpha z + \alpha 2\pi {\rm i}  = \Phi(z) + 2\pi {\rm i} + \gamma.\]
Thus, $\phi$ conjugates $h_{\Saff_1}$ to $h_{\Saff_2}$. 

This shows that $h_{\Saff_1}$ is holomorphically conjugate to $h_{\Saff_2}$ if and only if the marked tori $(\Sp_1;A_1,B_1)$ and $(\Sp_2;A_2,B_2)$ are biholomorphic. We deduce from Theorem~\ref{thm:main} that this is the case if and only if the geodesic flows on $\Saff_1$ and $\Saff_2$ are holomorphically conjugate.

\subsection{Topological conjugacy}

Let us first assume that the marked tori $(\Sp_1;A_1,B_1)$ and $(\Sp_2;A_2,B_2)$ are real-affine equivalent, i.e., there is a real-affine diffeomorphism $\phi:\Sp_1\to \Sp_2$ sending $A_1$ to $A_2$ and $B_1$ to $B_2$. Then, $\phi:\Sp_1\to \Sp_2$ lifts to a $\R$-linear homeomorphism $L:\C\to \C$ which sends $2\pi{\rm i}$ to $2\pi {\rm i} + \gamma$ for some $\gamma\in \Gamma_2$. By linearity of $L$, we have that 
\[\forall z\in \C, \quad L(z+2\pi {\rm i}) = L(z)+L(2\pi {\rm i}) = L(z) + 2\pi {\rm i} \gamma.\]
As a consequence, $\phi$ conjugates $h_{\Saff_1}$ to $h_{\Saff_2}$.

Let us now assume that $\phi:\Sp_1\to \Sp_2$ is a homeomorphism sending $A_1$ to $A_2$ and conjugating $h_{\Saff_1}$ to $h_{\Saff_2}$. Let $\Phi:\C\to \C$ be its lift fixing $0$. Let $(\mu_1,\nu_1)$ be a pair of generators of $\Gamma_1$ and set $\mu_2:=\Phi(\mu_1)$, $\nu_2:=\Phi(\nu_1)$.
Since $\phi$ conjugates $h_{\Saff_1}$ to $h_{\Saff_2}$,
\[\exists \gamma\in \Gamma_2,\quad \forall z\in \C, \quad \Phi(z+2\pi{\rm i}) = \Phi(z) + 2\pi{\rm i} + \gamma.\]
The orbit of $A_1$ under iteration of $h_{\Saff_1}$ is either dense in $\Sp_1$, dense in a topological circle contained in $\Sp_1$, or periodic. In all case, there are sequences $(k_j)_{j\geq 0}$, $(m_j)_{j\geq 0}$ and $(n_j)_{j\geq 0}$ such that 
\[k_j 2\pi {\rm i} - m_j \mu_1 - n_j \nu_1\underset{j\to +\infty} \longrightarrow 0.\]
Then, 
\begin{equation}\label{eq:limit1}
k_j (2\pi{\rm i} + \gamma) - m_j \mu_2 - n_j\nu_2 = \Phi(k_j 2\pi {\rm i} - m_j \mu_1 - n_j \nu_1) \underset{j\to +\infty} \longrightarrow 0.
\end{equation}
Now, let $L:\C\to \C$ be the $\R$-linear homeomorphism defined by 
\[\forall (x,y)\in \R^2,\quad L(x\mu_1 + y\nu_1) := x\mu_2+y\nu_2.\]
By $\R$-linearity of $L$, we have that
\begin{equation}\label{eq:limit2}
k_jL(2\pi {\rm i}) - m_j \mu_2 - n_j\nu_2 = L(k_j 2\pi {\rm i} - m_j \mu_1 - n_j \nu_1) \underset{j\to +\infty} \longrightarrow 0.
\end{equation}
Subtracting \eqref{eq:limit1} from \eqref{eq:limit2}, we deduce that 
\[k_j\bigl(L(2\pi {\rm i})  - 2\pi {\rm i} - \gamma\bigr) \underset{j\to +\infty} \longrightarrow 0\quad \text{and thus}\quad L(2\pi {\rm i}) = 2\pi {\rm i} + \gamma.\]
As a consequence, $L:\C\to \C$ projects to a real-affine homeomorphism $\Sp_1\to \Sp_2$ which sends $A_1$ to $A_2$ and $B_1$ to $B_2$. Thus, the marked tori $(\Sp_1;A_1,B_1)$ and $(\Sp_2;A_2,B_2)$ are real-affine equivalent.

\bibliographystyle{alpha}
\bibliography{bib}

\end{document}